\documentclass{svmult}
\usepackage{amsmath}
\usepackage{amsfonts}%
\usepackage{amssymb}
\usepackage{makeidx}         
\usepackage{graphicx}        
\usepackage{multicol}        
\usepackage[bottom]{footmisc}

\makeindex             

\begin{document}

\title*{On Rank Problems for Planar Webs and Projective Structures}

\author{Vladislav V. Goldberg\inst{1}\and
Valentin V. Lychagin\inst{2}}
\institute{New Jersey Institute of Technology, Newark, NJ, USA
\texttt{vladislav.goldberg@gmail.com} \and University of Tromso, Tromso, Norway
 \texttt{Valentin.Lychagin@matnat.uit.no}}
%
%
\maketitle

\begin{abstract}
We present some old and recent results on rank problems and  linearizability of geodesic planar webs.
\end{abstract}

\section{Introduction}
In this paper we continue our studies of geodesic planar webs \cite{GL08a}.

We give a modification of the Abel's elimination method. This method allows one to find  all abelian relations admitted by a planar web and therefore to determine the rank of the web.  It requires to solve step-by-step a series of ordinary differential equations. In \cite{P05} (see also \cite{P04}) the same modification is given by a little bit different approach.

On the other hand, we present the method of finding the web rank by means of differential invariants of the web, i.e., the determination of the web rank without solving  the  differential equations. Pantazi \cite{Pa38} found some necessary and sufficient conditions for a planar web to be of maximum rank. The paper \cite{Pa38} was followed by the papers \cite{Pa40}\ and \cite{Mi41}. Pirio in \cite{P04} presented a more detailed exposition of results of Pantazi in \cite{Pa38} and \cite{Pa40}\ and Mih\u{a}ileanu in \cite{Mi41}. The characterization of webs of maximal  rank in \cite{Pa38} and \cite{Mi41} is not given in terms of the web invariants.

We give also an alternative construction (the previous one was given in \cite{GL08a}) of the unique projective structure associated with a planar 4-web. Note that Theorem 7 was first proved in \cite{BB38} (see \S 29, p. 246) and that the result in \cite{BB38} was recently generalized in \cite{P08} for any dimension. Our method exploits differential forms and  gives an explicit formula  for the projective connection. Remark that this method, as well as one in (\cite{GL08a}), can be used in any dimension. Presence of the projective structure allows us to connect a differential invariant (which we call the Liouville tensor) with any planar 4-web. This tensor gives a criterion for linearizability of geodesic planar webs (cf. \cite{AGL04}).

\section{Planar Webs}

All constructions in the paper are local, and we do not specify
domains in which they are valid. Functions, differential forms,
etc. are real and of class $C^{\infty }$.

\emph{A planar $d$-web} is given by $d$ one-dimensional foliations
in the plane which are in general position, i.e., the directions
corresponding to different foliations are distinct. The local
diffeomorphisms of the plane act in the natural way on $d$-webs,
and they say that two $d$-webs are (locally) \emph{equivalent} if
there exists a local diffeomorphism which sends one $d$-web to
another.

Because all $2$-webs are locally equivalent, we begin with
$3$-webs.

A 3-web can be defined either by three differential $1$-forms, say,
$\omega_{1}, \omega_{2}, \omega_{3}$, where
$$
\omega_{1}\wedge \omega_{2}\neq 0, \;\; \omega_{2}\wedge
\omega_{3}\neq 0, \;\; \omega_{1}\wedge \omega_{3}\neq 0,
$$
or by the first integrals of the foliations, say,
$f_{1},f_{2},f_{3}$, where
$$
df_{1}\wedge df_{2}\neq 0, \;\; df_{2}\wedge df_{3}\neq 0,\;\;
df_{1}\wedge df_{3}\neq 0.
$$
The above functions $f_{1},f_{2},f_{3}$ are called
\emph{web functions}.

Remark that the web functions are defined up to \emph{gauge
transformations}
$$
f_{i}\mapsto \Phi_{i}(f_{i}),
$$
where $i=1,2,3,$ and $ \Phi_{i}:\mathbb{R}\rightarrow \mathbb{R}$
are local diffeomorphisms of the line.

The implicit function theorem states that there is a relation
$$
W(f_{1},f_{2},f_{3})=0
$$
for these functions.

The above relation is called (see, for example, \cite{B55}) the
\emph{web equation}.

Any pair of functions in this equation is locally
indistinguishable and can be viewed as local coordinates on the
plane.

Keeping in mind this observation, we consider a space
$\mathbb{R}^{3}$ with coordinates $u_{1}, u_{2}, u_{3}$ and
two-dimensional surface
$$
\Sigma\subset \mathbb{R}^{3}
$$
given by the equation
$$
W(u_{1}, u_{2}, u_{3})=0.
$$
We say that $\Sigma\subset \mathbb{R}^{3}$ is a \emph{web surface}
if any two functions $u_{i}, u_{j}$ are local coordinates on
$\Sigma$.

In the case of $d$-webs one can choose $d$ local first integrals
of the corresponding foliations, say,
$f_{1},f_{2},f_{3}...,f_{d}$, which are also called \emph{web
functions}. They define a map
$$
\sigma:\mathbb{R}^{2}\rightarrow \mathbb{R}^{d},
$$
where
$$
\sigma : (x,y)\in \mathbb{R}^{2}\mapsto
(u_{1}=f_{1}(x,y),....,u_{d}=f_{d}(x,y))\in \mathbb{R}^{d}.
$$
The image $\Sigma $ of this map is a two-dimensional surface in
$\mathbb{R}^{d}.$ Remark that any pair of functions $u_{i}, u_{j}$
are local coordinates on $\Sigma $.

From this point of view, the local theory of planar $d$-webs is
just a geometry of web surfaces in $\mathbb{R}^{d}$ considered
with respect to the gauge transformations.

\section {Basic Constructions}
Let us begin with 3-webs, and let differential $1$-forms
$\omega_{1}, \omega_{2}, \omega_{3}$ define such a web. These
forms are determined up to multipliers $\omega_{i} \leftrightarrow
\lambda_{i} \omega_{i}$, where $\lambda_{i}$ are smooth
nonvanishing functions. Hence, these forms can be normalized in
such a way that
\begin{equation}
\omega _{1}+\omega _{2}+\omega _{3}=0 \label{3-web normalization}
\end{equation}%
with only possible scaling $\omega_{i} \leftrightarrow \lambda
\omega_{i}$.

One can prove that in this case there is a unique differential $1$%
-form $\gamma $ such that the so-called \emph{structure equations}
\begin{equation}
d\omega _{i}=\omega _{i}\wedge \gamma \label{web structure equations}
\end{equation}%
hold for all $i=1,2,3$ (see \cite{AGL04}).

The form $\gamma $ determines the\emph{ Chern connection} $\Gamma$
in the cotangent bundle $T^{\ast }M$ with the following covariant
differential:
\begin{equation*}
d_{\Gamma }:\omega _{i}\longmapsto -\omega _{i}\otimes \gamma .
\end{equation*}%
The curvature of this connection is equal to
\begin{equation*}
R_{\Gamma }:\omega _{i}\longmapsto -\omega _{i}\otimes d\gamma .
\end{equation*}%
If we write
\begin{equation*}
d\gamma =K\omega _{1}\wedge \omega _{2},
\end{equation*}%
then the function $K$ is called the \emph{curvature function} of
the $3$-web.

Note that the curvature form $d\gamma $ is an invariant of the
$3$-web while the curvature function $K$ is a relative invariant
of the web of weight two.

Let $\left\langle \partial _{1},\partial _{2}\right\rangle$ be the
basis dual to $\left\langle \omega_{1}, \omega_{2}\right\rangle.$
We put $\partial_{3}=\partial_{2}-\partial_{1}.$ Then leaves of
the $3$-web are trajectories of the vector fields $\partial_{2},
\partial_{1},$ and $\partial _{3}.$

The form $\gamma$ can be decomposed as follows:

\begin{equation*}
\gamma =g_{1}\omega _{1}+g_{2}\omega _{2},
\end{equation*}%
where $g_{1}$ and $g_{2}$ are smooth functions.

Moreover, in this case one has (see \cite{GL06})
\begin{equation}
\lbrack \partial _{1},\partial _{2}]=-g_{2}\partial _{1} +
g_{1}\partial _{2} \label{commutator}
\end{equation}%
and
\begin{equation}
K=\partial_{1}\left( g_{2}\right) -\partial_{2}\left( g_{1}\right). \label{curvature equation for web}
\end{equation}%
Remark also that the covariant derivatives with respect to the Chern
connection have the form
$$
\nabla_{X} (\omega_{i})=-\gamma (X) \omega_{i}
$$
and
$$
\nabla_{X} (\partial_{i})=-\gamma (X) \partial_{i}.
$$
It shows that the leaves of all three foliations are geodesic with
respect to the Chern connection.

Let $d_{\nabla }:\Omega ^{1}(M)\rightarrow \Omega ^{1}\left( M\right)
\otimes \Omega ^{1}\left( M\right) $ be the covariant differential with
respect to the Chern connection.

The induced connection in the tangent bundle gives the differential
$$
d_{\nabla
}^{\ast }:\mathcal{D}\left( M\right) \rightarrow \mathcal{D}\left( M\right)
\otimes \Omega ^{1}\left( M\right),
$$
where
$$
d_{\nabla }:\partial _{i}\rightarrow \partial _{i}\otimes \gamma .
$$

In a similar way the Chern connection induces the covariant differential in the tensor bundles.

Let us denote by
$
\Theta ^{p,q}\left( M\right) =\left( \mathcal{D}\left( M\right)
\right) ^{\otimes p}\otimes \left( \Omega ^{1}\left( M\right) \right)
^{\otimes q}
$
the module of tensors of type $\left( p,q\right) .\;$

Then the covariant differential
\begin{equation*}
d_{\nabla }:\Theta ^{p,q}\left( M\right) \rightarrow \Theta
^{p+1,q}\left( M\right)
\end{equation*}%
acts as follows:
\begin{equation*}
d_{\nabla }:u\partial _{j_{1}}\otimes \cdots \otimes \partial
_{j_{p}}\otimes \omega _{i_{1}}\otimes \cdots \otimes \omega
_{i_{q}}\longmapsto \partial _{j_{1}}\otimes \cdots \otimes \partial
_{j_{p}}\otimes \omega _{i_{1}}\otimes \cdots \otimes \omega _{i_{q}}\otimes
\left( du+\left( p-q\right) \gamma u\right)
\end{equation*}%
where $u\in C^{\infty }\left( M\right) .$

We say that $u$ is of weight $k=q-p$ and call the form
\begin{equation}
\delta ^{(k)}\left( u\right) \overset{\text{def}}{=}du-k u\gamma
\label{covariant differential Vadim}
\end{equation}%
the \emph{covariant differential }of $u.$

Decomposing the form $\delta ^{(k) }\left( u\right) $ in the
basis $\{\omega _{1},\omega _{2}\},$ we obtain
\begin{equation*}
\delta ^{\left( k\right) }\left( u\right) =\delta _{1}^{\left( k\right)
}\left( u\right) ~\omega _{1}+\delta _{2}^{\left( k\right) }\left(
u\right) ~\omega _{2},
\end{equation*}%
where
$$
\delta _{i}^{\left( k\right) }\left( u\right)=\partial _{i}\left( u\right) -\left( k\right) g_{i}u
$$
are the covariant derivatives of $u$ with respect to the Chern connection, $i=1,2$.

Note that $\delta _{1}^{\left( k\right) }\left( u\right) $ and $\delta
_{2}^{\left( k\right) }\left( u\right) $ are of weight $k+1.$

 One can check that the covariant derivatives satisfy the classical Leibnitz rule

\begin{equation*}
\delta _{i}^{\left( k+l\right) }\left( uv\right) =\delta _{i}^{\left(
k\right) }\left( u\right) ~v+u~\delta _{i}^{\left( l\right) }\left( v\right)
\end{equation*}
if $u$ is of weight $k$ and $v$ is of weight $l$, and the following commutation relation:

\begin{equation}
\delta _{2}^{\left( s+1\right) }\circ \delta _{1}^{\left( s\right) }-\delta
_{1}^{\left( s+1\right) }\circ \delta _{2}^{\left( s\right) }=sK.
\label{covariant commutator}
\end{equation}

Note that the curvature $K$ is of weight two.

In what follows, we shall omit the superscript indicating the weight in the
cases when the weight is known.


For the general $d$-web defined by differential $1$-forms $
\omega_{1}, \omega_{2}, \omega_{3}, \dots, \omega_{d}$ we
normalize $\omega_{1}, \omega_{2}, \omega_{3}$ as above and choose
$\omega_{i}$ for $i\geq 4$ in such a way that the normalizations
\begin{equation}
a_{i}\omega_{1} + \omega_{2} + \omega_{i+2}=0 \label{normalized
eqs for d-web}
\end{equation}%
hold for $i=1,...,d-2,$ with $a_{1}=1.$

Note that $a_{i} \neq 0,1$ for $i \geq 2.$

Moreover, for any fixed $i,$\ the value $a_{i}\left( x\right) ,$
of the function $a_{i}$ at the point $x$ is the cross-ratio of the
four straight lines in the cotangent space $T_{x}^{\ast }$
generated by the covectors $\omega_{1,x}, \omega_{2,x},
\omega_{3,x}$, and $\omega _{i+2,x},$ and therefore it is a web
invariant. The functions $a_{i}$ are called the \emph{basic
invariants} (cf. \cite{G04} or \cite{G88}, pp. 302--303) of the
web.

Because of locality of our consideration, one can  choose a
function $f$ in such a way that $\omega_{3}= df$ and find
coordinates $x,y$  such that $\omega_{1} \wedge dx=0$ and
$\omega_{2} \wedge dy = 0$.

Let also $\omega _{i+3} \wedge dg_{i}$ $=0,$  for some functions
$g_{i}(x,y)$, $i=1,...,d-3$.

Then $\omega_{1} = -f_{x} dx$ and $\omega_{2} = -f_{y} dy$.

The dual basis $\left\{\partial_{1},\partial_{2}\right\}$ has the
form
\begin{equation*}
\partial_{1} = -f_{x}^{-1} \partial_{x},\ \ \ \partial_{2} = -f_{y}^{-1} \partial_{y},
\end{equation*}%
and the connection form is
\begin{equation*}
\gamma =-H \omega_{3},
\end{equation*}%
where
\begin{equation*}
H= \frac{f_{xy}}{f_{x}~f_{y}}
\end{equation*}%
(see \cite{GL06}).

The curvature function has the following explicit expression:
\begin{equation}
K  = -\displaystyle\frac{1}{f_{x} f_{y}} \Biggl(\log
\displaystyle\frac{f_{x}}{f_{y}}\Biggr)_{xy},
     \label{3-web curvature}
\end{equation}%
and the basic invariants have the form%
\begin{equation*}
a_{i} = \frac{f_{y}g_{i,x}}{f_{x}g_{i,y}}
\end{equation*}%
for $i=1,...,d-3.$

Note that if a three-web $W_3$ is given by a web equation $W (u_1,
u_2, u_3) = 0$, then the curvature $K$ is expressed as follows
(see \cite{B55}, \S 9):
\begin{equation}
K  =A_{12} + A_{23} + A_{31},
     \label{3-web curvature for W(u1,u2,u3)=0}
\end{equation}
where
$$
A_{rs} = \displaystyle\frac{1}{W_{r} W_{s}}
\frac{\partial^2}{\partial u_r \partial u_s} \log
\displaystyle\frac{W_{r}}{W_{s}},
$$
and subscripts $r$ and $s$ mean the partial derivatives of the
function with respect to the variables $u_r$ and $u_s$, where $ r,
s = 1, 2, 3$.

Recall that a planar $d$-web is said to be (locally)
\emph{parallelizable} if it is (locally) equivalent to a $d$-web
of parallel straight lines in the affine plane.

It is  known (see, for example, \cite{B55}, $\S 8$) that
 \emph{a planar} $3$\emph{-web is locally parallelizable if and only if }$K=0.$

For planar $d$-webs, $d\geq 4,$ the following statement holds (cf.
\cite{G04} or \cite{G88}, Section 7.2.1 for $d=4$): \emph{a planar
$d$-web  $\left\langle \omega _{1},\omega _{2},\omega _{3},\omega
_{4},...,\omega _{d}\right\rangle $ is locally parallelizable if
and only if its $3$-subweb $\left\langle \omega
_{1},\omega _{2},\omega _{3}\right\rangle $ is locally parallelizable $($%
i.e., $K=0)$, and all basic invariants $a_{i}$ are constants}.

\section{Rank}

We begin with an interpretation of the classical Abel addition
theorem (\cite{Ab29}) in terms of planar webs (cf. \cite{B55}).

Let us consider linear webs on the affine plane, i.e., such planar
webs leaves of which are straight lines. There is an elegant
method to construct such webs. Take a straight line $rx+sy=1$ on
the affine plane and assume that the coefficients $(r,s) \in
\mathbb{R}^2$ satisfy an algebraic equation $P_{d}(r,s)=0$ of
degree $d$.

Given $(x,y)$, then the system
$$
\renewcommand{\arraystretch}{1.5}
\left\{
\begin{array}{ll}
rx + sy &=1, \\
P_{d}(r,s)&=0
\end{array}
\right.
\renewcommand{\arraystretch}{1}
$$
has at most $d$ roots.

Assume that in a domain on the plane $(x,y)$ the above system has
exactly $d$ roots. Then in this domain we have a linear $d$-web.

Take now a cubic polynomial
$$
P_{3}(s,t) = s^{2} - 4r^{3} - g_{2}r -g_{3},
$$
where $g_{2}$ and $g_{3}$ are constants.

Then the system
$$
\renewcommand{\arraystretch}{1.5}
\left\{
\begin{array}{ll}
rx + sy=1, \\
s^{2} -4r^{3} - g_{2}r - g_{3}=0
\end{array}
\right.
\renewcommand{\arraystretch}{1}
$$
in the domain
$$
x^{4} - 24 xy^{2} - 12 g_{2} y^{4} > 0,\ y \neq 0,
$$
has three distinct real roots and consequently three pairwise
distinct straight lines $(r_{i} (x, y), s_{i}(x, y))$, passing
through the point $(x, y)$. In other words, we have a linear
$3$-web.

Assume that
$$
g_{2}^{3}-27g_{3}^{2}\neq 0.
$$

Then the solutions of the equation $%
s^{2}-4r^{3}-g_{2}r-g_{3}=0$ can be parameterized by the
Weierstrass' elliptic function with the invariants $g_{2}$ and
$g_3$:
$$
r=\wp ( t) ,  s=\wp^{\prime }(t) .
$$
Hence, the roots $(r_{i} (x, y), s_{i} (x, y))$ correspond to
three solutions $(t_{i} ( x, y))$ of the equation
\begin{equation*}
\wp (t) x + \wp ^{\prime}(t) y - 1 = 0.
\end{equation*}%
Let us put
$$
f(t)=\wp (t) x + \wp^{\prime} (t) y - 1
$$
and compute the integral
\begin{equation*}
\int t\frac{f^{\prime} (t)} {f( t)} dt
\end{equation*}%
along the boundary of the period parallelogram of the Weierstrass function.
We get
\begin{equation*}
t_{1} (x, y) + t_{2} (x, y) + t_{3} (x, y) =\operatorname*{const}.
\end{equation*}%
This is the \emph{abelian relation}.

This relation can be understood geometrically if we note that, by
the construction, the functions $t_{1} \left(x,y\right), t_{2}
\left(x,y\right),$ and $t_{3} \left(x,y\right)$ are first
integrals of the corresponding $3$-web.

In more general case, let us consider an arbitrary planar $d$-web
defined by $d$ web functions
$$
f_{1}, \dots ,f_{d}.
$$

Then by \emph{abelian relation} we mean a relation
\begin{equation*}
F_{1} \left(f_{1}\right) + \dots + F_{d} \left(f_{d}\right)
=\operatorname*{const}.
\end{equation*}%
given by $d$ functions $(F_{1}, \dots , F_{d})$ of one variable.

We say that two abelian relations $\left(F_{1}, \dots ,
F_{d}\right)$ and $\left(G_{1}, \dots, G_{d}\right)$ are
\emph{equivalent} if
$$
F_{i}=G_{i}+\operatorname*{const}_{i},
$$
for all $i = 1, \dots , d$.

The set of equivalence classes of abelian relations
admits the natural vector space structure with respect to addition
$$
\left(F_{1}, \dots , F_{d}\right) +\left( G_{1}, \dots ,
G_{d}\right) =\left( F_{1} + G_{1}, \dots , F_{d} + G_{d}\right)
$$
and multiplication by numbers
$$
\alpha \left(F_{1}, \dots , F_{d}\right) =\left(\alpha F_{1},
\dots , \alpha F_{d}\right).
$$
The dimension of this vector space is called the \emph{rank} of the web.

In the case when $d$-web is defined by differential $1$-forms
$$
 \omega_{1}, \dots , \omega_{d},
$$
the differentiation of the abelian relation leads us to the
\emph{abelian equation}
\begin{equation*}
\lambda_{1} \omega_{1} + \dots + \lambda_{d} \omega_{d} = 0,
\end{equation*}%
for functions $\lambda_{1}, \dots , \lambda_{d}$ under the
condition that all differential $1$-forms $\lambda_{i} \omega_{i}$
are closed:
$$
d(\lambda_{i} \omega_{i}) = 0.
$$

The abelian equation is a system of the first-order linear PDEs for the functions $\left( \lambda
_{1},...,\lambda _{d}\right) ,$ and the rank of the web is the
dimension of the solution space.
\begin{example}
The following example illustrates the above constructions for $3$-webs.

Consider the $3$-web $W_{3}$ given by  web functions:
$$
x, y, f(x,y).
$$
Then
\begin{equation*}
\omega_{1} = -f_{x}dx, \;\; \omega _{2} = -f_{y}dy, \;\;\omega _{3}=df,
\end{equation*}%
and the condition
\begin{equation*}
\lambda_{1} \omega_{1} + \lambda_{2} \omega_{2} + \lambda_{3}
\omega_{3} = 0
\end{equation*}%
implies
\begin{equation*}
\lambda_{1} = \lambda_{2} = \lambda_{3}=\lambda.
\end{equation*}%
The abelian relations take now the form
$$
\renewcommand{\arraystretch}{1.3}
\left\{
\begin{array}{ll}
(\lambda f_{x})_{y}&=0, \\
(\lambda f_{y})_{x}&=0,\\
\lambda_{x}f_{y}-\lambda_{y}f_{x}&=0,\nonumber
\end{array}
\right.
\renewcommand{\arraystretch}{1}
$$
or
$$
\renewcommand{\arraystretch}{1.3}
\left\{
\begin{array}{ll}
   (\ln\lambda)_{x}=-(\ln f_{x})_{y},\\
   (\ln\lambda)_{y}=-(\ln f_{y})_{x}.
\end{array}
\right.
\renewcommand{\arraystretch}{1}
$$
The compatibility condition for this system has the form
$$
(\ln f_{x})_{xy} = (\ln f_{y})_{xy}
$$
or
$$
K = 0.
$$
So, we can conclude this consideration by the following statement:
\emph{rank of a $3$-web does not exceed one, and the rank equals
to one if and only if the $3$-web is parallelizable}.
\end{example}

\section{Abel's Method}

In this section we discuss the rank problem in the classical
setting. A method of finding the rank, or in other terms, a method
of solving abelian relations was proposed by Abel himself (see
\cite{Ab23}). This method is just a consistent elimination of the
functions from the abelian relation by using only differentiation.

Let us consider a planar $d$-web defined by web functions $f_{1},
\dots , f_{d}$ and the corresponding  abelian relation
\begin{equation} F_{1} \left(f_{1}\right) + \dots + F_{d} \left(f_{d}\right)
=\operatorname*{const}. \label{afe}
\end{equation}

Modifying Abel's method and adjusting it to equation (\ref{afe}), we
can explain it as follows:

\begin{description}
\item{a)} Taking the differential of equation (\ref{afe}), we get
\begin{equation}
F'_{1}\, df_{1} + F'_{2}\, df_{2} + \dots + F'_{d}\, df_{d}=0.
\label{differential}
\end{equation}

\item{b)} Taking the wedge product of (\ref{differential}) with
$df_1$, we eliminate $F_1$ and get the following equation:
\begin{equation}
 F'_{2} + J^{3 1}_{2 1} F'_{3} + \dots + J^{d 1}_{2 1} F'_{d} =0, \label{F1 eliminated}
\end{equation}
where
$$
J^{i j}_{k l}=\frac{\partial(f_{i},f_{j})}{\partial(f_{k},f_{l})}
$$
is the Jacobian of the functions $f_{i}, f_{j}$ with respect to
functions $f_{k}, f_{l}.$

\item{c)} Taking the wedge product of the  differential of
(\ref{F1 eliminated}) with $df_2$, we eliminate $F_2$ and get the
following equation:
\begin{equation}
 J^{3 1}_{2 1} F''_3 + a_3 F'_3 + \dots = 0, \label{F2 eliminated}
\end{equation}
where $a_3 $ is a certain function.

\item{d)} Divide equation (\ref{F2 eliminated}) by the first
coefficient and take the differential of the obtained equation; if
$F''_3 $ appears in differentiation, take its value from
equation (\ref{F2 eliminated}). This gives the following equation:
\begin{equation}
  F'''_3  df_3  + b_3 F'_3 (f_3) + \dots =0, \label{differential3}
\end{equation}
where $b_3$ is a certain function.

\item{e)} Taking the exterior product of (\ref{differential3})
with $df_3$, we eliminate $F'''_3$ and get the following
equation:
\begin{equation}
   c_3  F'_3 (f_3) + \dots = 0, \label{F''3 eliminated}
\end{equation}
where $c_3$ is a certain function.

\item{f)} Dividing equation (\ref{F''3 eliminated}) by $c_3$ and
taking the wedge product of the differential of the obtained
equation and $df_3$, we eliminate the function $F_3$.

\item{g)} Use the procedure outlined above to eliminate the
functions $F_4, \dots, F_{d-1}$.

Finally, we obtain a linear differential equation with respect to
the function $F_d (f_d)$. This equation can be viewed as family of
homogeneous ordinary linear differential equations.

\item{h)} Substitute the solution $F_d (f_d)$ into (\ref{afe}) and
apply the outlined procedure to find another function, say,
$F_{d-1} (f_{d-1})$.
\end{description}

On Abel's elimination method as well as on less general method of
monodromy
see  \cite{P05}.

Below we give few examples of application of the Abel method.

\subsection{3-Webs}

Here we apply the Abel elimination method for $3$-webs to show
once more that a planar $3$-web admits a nontrivial abelian relation if and
only if the $3$-web is parallelizable.

Suppose that a $3$-web is given by the web functions $f(x, y), x,
y$ and let
\begin{equation}
F(f) + G (x)  +  H (y) = 0. \label{afe for 3-web}
\end{equation}
be an abelian relation.

Take the differential of (\ref{afe for
3-web}):
\begin{equation}
F'(f)\,df + G' (x)\,dx  +  H' (y)\,dy = 0, \label{afe for 3-web1}
\end{equation}
and the wedge product of (\ref{afe for 3-web1}) with $df$:
\begin{equation}
-f_y \,G' (x)  +  f_x \, H' (y) = 0. \label{afe for 3-web2}
\end{equation}
Then
\begin{equation}
G' (x)  - \frac{f_x}{f_y}\, H' (y) = 0. \label{afe for 3-web3}
\end{equation}
Taking the wedge product of the differential of (\ref{afe for
3-web3}) with $dx$, we get
\begin{equation}
 H'' (y) + \Biggl(\log \displaystyle\frac{f_x}{f_y}\Biggr)_y \, H' (y)= 0.
\label{afe for 3-web5}
\end{equation}
In order equation (\ref{afe for 3-web5}) has a  nontrivial
solution, it is necessary and sufficient that the function
$\Biggl(\log \displaystyle\frac{f_x}{f_y}\Biggr)_y$ does not
depend on $x$, i.e.,
\begin{equation}
 \Biggl(\log \displaystyle\frac{f_x}{f_y}\Biggr)_{xy}=0.
\label{afe for 3-web6}
\end{equation}
This means that $K = 0$, i.e., the 3-web is parallelizable.\\

\subsection{4-Webs of Rank Three}

Assume that a $4$-web is given by the following web functions:
$$
f =x +y,\; g =xy,\; x,\; y.
$$
We will apply the Abel elimination method to find all abelian
relations admitted by this web.

Let
\begin{equation}
F (f) + G (g) + H (x)  +  K (y) = 0 \label{afe3}
\end{equation}
be an abelian relation.

Taking the differential of (\ref{afe3}):
\begin{equation}
F' (f) \,df + G' (g) \,dg + H' (x) dx + K'(y)\,
dy=0, \label{afe31}
\end{equation}
and the wedge product of (\ref{afe31}) with $dy$, we eliminate $K' (y)$:
\begin{equation}
F'(f)  + y\, G' (g) + H' (x) = 0. \label{afe32}
\end{equation}

Once more, taking the differential of (\ref{afe32}):
\begin{equation}
F''(f)\, df + y\, G''(g)\, dg + G' (g)\, dy + H''(x) \,dx = 0, \label{afe33}
\end{equation}
and the wedge product of (\ref{afe33}) with $dg$, we eliminate $G''(g)$:
\begin{equation}
(x - y) \,F''(f) -y \, G' (g) +x/, H'' (x) = 0. \label{afe34}
\end{equation}
Using equation (\ref{afe32}), we eliminate $G' (g)$ in (\ref{afe34}):
\begin{equation}
(x - y) \,F''(f) + F' (f) + H' (x) + x H'' (x) = 0. \label{afe35}
\end{equation}

Dividing equation (\ref{afe35}) by $x$, taking the differential of the result and taking
the wedge product of the differential with $dx$, we eliminate $H (x)$ and arrive
at the equation
\begin{equation}
F'''(f) = 0.
\label{afe36}
\end{equation}

Up to an arbitrary constant, the solution of equation (\ref{afe36}) is
\begin{equation}
F (f) = a \, f^2 + b \, f,
  \label{afe37}
\end{equation}
where $a$ and $b$ are arbitrary constants.
By (\ref{afe37}), equation (\ref{afe35}) gives
\begin{equation}
x\, H'' (x) + H' (x) + 2ax = 0. \label{afe38}
\end{equation}
Up to an arbitrary constant, the solution of equation (\ref{afe38}) is
\begin{equation}
H (x) = - a\, x^2 - b\, x + k \log x,
  \label{afe39}
\end{equation}
where $k$ is an arbitrary constant.

By (\ref{afe37}) and (\ref{afe39}), equation (\ref{afe34}) gives
\begin{equation}
G' (g) = - a - \frac{k}{g}. \label{afe310}
\end{equation}
Up to an arbitrary constant, the solution of equation (\ref{afe310}) is
\begin{equation}
G (g) = - 2a\, g - k \log g,
  \label{afe311}
\end{equation}

Now by (\ref{afe37}), (\ref{afe39}) and (\ref{afe311}), we find from equation (\ref{afe3})
that
\begin{equation}
K (y) = - a\, y^2 - b\,y +k \, \log y. \label{afe312}
\end{equation}

Thus, the rank is equal to three, and  we have the following three independent abelian relations:
\begin{description}
\item{$(a = 0, b = 0, k = -1)$}
$$
x + y - f = 0;
$$
\item{$(a = - 1, b = 0, k = 0)$}
$$
x^2 + y^2 + (-f^2) + (2g) =0;
$$
\item{$(a = 0, b = -1, k = 0)$}
$$
\log x + \log y + (-\log g) =0.
$$
\end{description}

The fact that the rank of this web is three was also proved in
\cite{GL07} by use of differential invariants of webs.

\subsection{4-Webs of Rank Two}

Consider a $4$-web given by the following web functions:
$$
f =x^2+y^2,\; g =x+y,\; x,\; y.
$$
We will apply the Abel elimination method and find all abelian
relations admitted by this web.

Let
\begin{equation}
F (f) + G (g) + H (x)  +  K (y) = 0 \label{afe2}
\end{equation}
be an abelian relation.

Taking the differential:
\begin{equation}
F' (f) \,df + G' (g) \,dg + H' (x)\, dx + K'(y)\,
dy=0, \label{afe21}
\end{equation}
and the wedge product of (\ref{afe21}) with $dy$, we eliminate $K (y)$:
\begin{equation}
2x \,F'(f)  + G' (g) + H' (x) = 0. \label{afe22}
\end{equation}

Once more, taking the differential of (\ref{afe22}):
\begin{equation}
2x\, F''(f)\, df +  2F'(f)\, dx + G'' (g)\, dg + H''(x) \,dx = 0, \label{afe23}
\end{equation}
and the wedge product of (\ref{afe23}) with $dx$, we eliminate $H (x)$:
\begin{equation}
2(g^2 - f) \,F''(f)  + G'' (g) = 0. \label{afe24}
\end{equation}
Taking the wedge product of the differential of (\ref{afe24}) with
$dg$, we eliminate $G (g)$:
\begin{equation}
2(g^2 - f) \, F'''(f)  -2 F'' (f) = 0.
\label{afe25}
\end{equation}

Equation (\ref{afe25}) is equivalent to the system
\begin{equation}
\renewcommand{\arraystretch}{1.3}
\left\{
\begin{array}{ll}
-2 f \, F'''(f)  -2 F'' (f) = 0,\\
 F'''(f) = 0.
 \end{array}
\right.
\renewcommand{\arraystretch}{1}
 \label{afe26}
\end{equation}
Therefore, up to an additive constant,
$$
F (f) =k f,
$$
where $k$ is a constant.

Now it follows from (\ref{afe24}) that, up to an additive constant,
$$
G(g) = b g,
$$
where $b$ is a constant.

Equation (\ref{afe22}) implies that, up to an additive constant,
$$
H (x) = -k x^2-b x,
$$
and equation (\ref{afe2}) gives that
$$
K (y) = - k y^2 - b y.
$$.

Thus, the rank of the web is equal to two, and  we have the following two basic abelian relations:
\begin{description}
\item{$(k = 0, b = 1)$}
$$
(x+y) + (- x) + (- y) = 0,
$$
\item{$(k = 1, b = 0)$}
$$
(x^2 + y^2)+ (-x^2)+ (-y^2)=0.
$$
\end{description}
In \cite{GL07} by use of differential invariants of webs, it was
shown that this $4$-web is of rank two.

\subsection{4-Webs of Rank One}

Assume that a $4$-web is given by the following web functions:
$$
f = \frac{(x-y)^2}{x},\; g = \frac{(x-y)^2}{y},\; x,\; y.
$$
We will apply the Abel elimination method to find all abelian
relations admitted by this web.

Let
\begin{equation}
F (f) + G (g)  + H (x) + K (y) = 0 \label{afe4}
\end{equation}
be an abelian relation.

Taking the differential of (\ref{afe4}):
\begin{equation}
F' (f) \,df + G' (g) \,dg + H' (x) dx + K'(y)\,
dy=0, \label{afe41}
\end{equation}
and the wedge product of (\ref{afe41}) with $df$, we eliminate $F(f)$:
\begin{equation}
G'(g)  -\displaystyle \frac{2xy^2}{(x-y)^3} H' (x) - \frac{(x+y)y^2}{(x-y)^3} K' (y) = 0. \label{afe42}
\end{equation}

Once more, taking the differential of (\ref{afe42}):
\begin{equation}
\renewcommand{\arraystretch}{2}
\begin{array}{ll}
G''(g)\, dg - \displaystyle\frac{2xy^2}{(x-y)^3} H'' (x) -\displaystyle \frac{2y(2x+y)(-y dx +x dy)}{(x-y)^3} H' (x) \\ - \displaystyle\frac{(x+y)y^2}{(x-y)^3} K'' (y) - \frac{2y(x+2y)(-y dx +x dy)}{(x-y)^3} K' (y) = 0, \label{afe43}
\end{array}
\renewcommand{\arraystretch}{1}
\end{equation}
and the wedge product of (\ref{afe43}) with $dg$, we eliminate $G (g)$:
\begin{equation}
H''(x)  + \displaystyle\frac{2x+y}{x(x+y)} \, H' (x) +\frac{y}{x} \, K'' (y) +\frac{x+2y}{x(x+y)} \, K' (y) = 0. \label{afe44}
\end{equation}
Taking the wedge product of the differential of (\ref{afe44}) with
$dx$, we eliminate $H'' (x)$:
\begin{equation}
H'(x)  - \frac{(x+y)^2 y}{x} K''' (y) - \frac{(2x+3y)(x+y)}{x}\, K'' (y) - K' (y) = 0.
\label{afe45}
\end{equation}

Taking the wedge product of the differential of (\ref{afe45}) with
$dx$, we eliminate $H' (x)$:
\begin{equation}
   (x+y) y K^{iv} (y) + 3 (x+2y)\, K''' (y) + 6 K'' (y) = 0.
\label{afe46}
\end{equation}

Equation (\ref{afe46}) is equivalent to the system
\begin{equation}
\renewcommand{\arraystretch}{1.3}
\left\{
\begin{array}{ll}
y\, K^{iv}(y)  + 3x K''' (x) = 0,\\
y^2\, K^{iv}(y) + 6y\, K'''(y) + 6 K''(y) = 0.
 \end{array}
\right.
\renewcommand{\arraystretch}{1}
 \label{afe47}
\end{equation}
It follows from (\ref{afe47}) that
\begin{equation}
   y \,K''' (y) + 2 K'' (y) = 0.
\label{afe48}
\end{equation}

Up to an additive constant, the solution of (\ref{afe48}) is
\begin{equation}
   K (y) = - k \log y + b y,
\label{afe49}
\end{equation}
where $k$ and $b$ are arbitrary constants.

It follows from (\ref{afe45}) and (\ref{afe49}) that
\begin{equation}
  H' (x) = \frac{k}{x} + b.
\label{afe410}
\end{equation}

Up to an additive constant, the solution of (\ref{afe410}) is
\begin{equation}
   H (x) =  k \log x + b x.
\label{afe411}
\end{equation}

It follows from (\ref{afe42}), (\ref{afe49}) and (\ref{afe411}) that
\begin{equation}
  G' (g) = -\frac{k}{g} + \frac{by (3x+y)}{g(x-y)}.
\label{afe412}
\end{equation}

Equation (\ref{afe412}) implies that
\begin{equation}
   b = 0
\label{afe413}
\end{equation}
and that, up to an additive constant,
\begin{equation}
   G (g) = - k \log g, \;\; H (x) = k \log x, \;\; K (y) = - k \log y.
\label{afe414}
\end{equation}

Finally, equations (\ref{afe4}) and (\ref{afe414}) give
\begin{equation}
   F (f) =  k \log f,
\label{afe415}
\end{equation}
and the 4-web admits only one independent abelian relation
\begin{equation}
  \log f - \log g + \log x - \log y = 0.
\label{afe416}
\end{equation}

By use of differential invariants of webs introduced in \cite{GL07}, one can show that this $4$-web is of rank one.

\subsection{4-Webs of Rank Zero}
Assume that a $3$-web is given by the following web functions:
$$
f = (x +y) e^x,\; g = xy, \; x,\; y.
$$
We will apply the Abel elimination method to show that this web
admits no abelian relations.

Let
\begin{equation}
F (f) + G (g) + H (x) +  K (y) = 0 \label{afe5}
\end{equation}
be an abelian relation.

Taking the differential of (\ref{afe5}):
\begin{equation}
F' (f) \,df + G'(g) \, dg + H' (x) \,dx + K'(y)\,dy=0, \label{afe51}
\end{equation}
and the wedge product of (\ref{afe51}) with $df$, we eliminate $F (f)$:
\begin{equation}
[(x + y)x + x - y]\, G' (g)- H'(x)  + (1 + x + y) \,K' (y) = 0. \label{afe52}
\end{equation}

Once more, taking the differential of (\ref{afe52}):
\begin{equation}
\renewcommand{\arraystretch}{1.3}
\begin{array}{ll}
&[(x + y)x + x - y]\, G'' (g)\, dg - H''(x)\, dx +  (1 + x + y) \,K''(y)\, dy \\+& [(2x + y + 1) \, dx + (x - 1) \, dy] \, G' (g) + (dx + dy)\, K' (y) = 0, \label{afe53}
\end{array}
\renewcommand{\arraystretch}{1}
\end{equation}
and the wedge product of (\ref{afe53}) with $dx$, we eliminate $H (x)$:
\begin{equation}
x\, [(x + y)x + x - y]\, G'' (g) + (x - 1) \, G'(g) + K''(y)  + K' (y) = 0. \label{afe54}
\end{equation}

Taking the differential of (\ref{afe54}):
\begin{equation}
\renewcommand{\arraystretch}{1.3}
\begin{array}{ll}
x\, [(x + y)x + x - y]\, G''' (g)\, dg + (x - 1) \, G''(g)\, dg\\ + K'''(y)\, dy  + K'' (y)\, dy + G'' (g) \, dx = 0, \label{afe55}
\end{array}
\renewcommand{\arraystretch}{1}
\end{equation}
and the wedge product of (\ref{afe55}) with $dy$, we eliminate $K (y)$:
\begin{equation}
g\, (x^2 + g + x - \frac{g}{x}]\, G''' (g) + (3g - \frac{2g}{x} + 3x^2 + 2x + 1) \, G''(g) = 0. \label{afe56}
\end{equation}

Equation (\ref{afe56}) is equivalent to the system
\begin{equation*}
\renewcommand{\arraystretch}{1.3}
\left\{
\begin{array}{ll}
g^2 G''' (g) + (3g+1)\,G''(g) = 0, \\
G'' (g) = 0, \\
G'''(g) = 0,
\end{array}
\right.
\renewcommand{\arraystretch}{1}
\end{equation*}
i.e., to the equation $G'' (g) = 0$. Up to an arbitrary constant, the solution of the latter equation is
$G = ag$, where $a$ is an arbitrary constant.

If $G = ag$, then equation (\ref{afe54}) becomes
$$
K'' (y) + K' (y) + a\, (x-1) = 0.
$$
It follows that $a = 0$ and $G (g) = 0$. The equation for $K (y)$ becomes
$K'' (y) + K' (y) = 0$. Up to an arbitrary constant, its solution is $K (y) = -b e ^{-y}$,
where $b$ is an arbitrary constant.

Now equation (\ref{afe52}) becomes
$$
H' (x)= b(1+y) e^{-y} + b\,xe^{-y}.
$$
It follows that $b = 0$ and $H' (x) = 0$. Hence, up to an arbitrary constant,  $H (x) = 0$ and $K (y) = 0$.

Finally equation (\ref{afe5}) implies that $F (f) = 0$.

Thus, the web under consideration
admits no abelian relations.

\section{Abelian Differential Equations}

 In this section we discuss properties of abelian equations.

 Recall that the abelian equation for a planar $d$-web given by differential $1$-forms
$$
\omega _{1},...,\omega _{d}
$$
is a first-order PDE system for functions $\lambda _{1},...,\lambda _{d}$ of the form
\begin{eqnarray*}
&&\lambda _{1}\omega _{1}+\cdots +\lambda _{d}\omega _{d}=0, \\
&&d\left( \lambda _{1}\omega _{1}\right) =\cdots =d\left( \lambda _{d}\omega
_{d}\right) =0.
\end{eqnarray*}

Let us write down the abelian equation in more explicit form. To this end, we choose a $3$-subweb, say, the $3$-web given by
$$
\omega_{1},\omega _{2},\omega _{3},
$$
and normalize the $d$-web as it was done earlier:
\begin{equation*}
a_{1}\omega _{1}+\omega _{2}+\omega _{3}=0,\ \;a_{2}\omega _{1}+\omega
_{2}+\omega _{4}=0,....,\;a_{d-2}\omega _{1}+\omega _{2}+\omega _{d}=0,
\end{equation*}
with $a_{1}=1$ and  $\ d\omega _{3}=0.$

It is easy to see that, if $i\leq 3$, then, due to the structure equations, we get
\[
d\left( \lambda \omega _{i}\right) =d\lambda \wedge \omega _{i}+\lambda
d\omega _{i}=\left( d\lambda -\lambda \gamma \right) \wedge \omega _{i}
\]
or
\[
d\left( \lambda \omega _{i}\right) =\delta \left( \lambda \right) \wedge
\omega _{i},
\]
if we consider $\lambda $ as a function of weight one.

Assuming that all $\lambda_{i}$ are functions of weight one and the functions $a_{i}$ are of weight $0$, we get
\begin{eqnarray*}
\renewcommand{\arraystretch}{1.3}
d\left( \lambda _{1}\omega _{1}\right)  &=&-\delta _{2}\left( \lambda
_{1}\right) \omega _{1}\wedge \omega _{2},\  \\
d\left( \lambda _{2}\omega _{2}\right)  &=&\delta _{1}\left( \lambda
_{1}\right) \omega _{1}\wedge \omega _{2}, \\
d\left( \lambda _{3}\omega _{3}\right)  &=&\left( \delta _{2}\left( \lambda
_{3}\right) -\delta _{1}\left( \lambda _{3}\right) \right) \omega _{1}\wedge
\omega _{2}, \\
d\left( \lambda _{i}\omega _{i}\right)  &=&\left( \delta _{2}\left(
a_{i-2}\lambda _{i}\right) -\delta _{1}\left( \lambda _{i}\right) \right)
\omega _{1}\wedge \omega _{2}
\renewcommand{\arraystretch}{1}
\end{eqnarray*}%
for $i=4,...,d.$

The normalization condition $\sum_{1}^{d}\lambda _{i}\omega _{i}=0$ implies
that
\begin{eqnarray*}
\renewcommand{\arraystretch}{1.3}
\lambda _{1} &=&a_{1}u_{1}+\cdots +a_{d-2}u_{d-2}, \\
\lambda _{2} &=&u_{1}+\cdots +u_{d-2},
\renewcommand{\arraystretch}{1}
\end{eqnarray*}
where
$$
u_{1}=\lambda_{3}, \dots , u_{d-2}=\lambda_{d}.
$$

Therefore the abelian equation can be written in the explicit form as the following PDE system:
\begin{align*}
\renewcommand{\arraystretch}{1.3}
& \Delta _{1}\left( u_{1}\right) =\cdots =\Delta _{d-2}\left( u_{d-2}\right)
=0, \\
& \delta _{1}\left( u_{1}\right) +\cdots +\delta _{1}\left( u_{d-2}\right)
=0,
\renewcommand{\arraystretch}{1}
\end{align*}%
where $\Delta _{i}=\delta _{1}-\delta _{2}\circ a_{i}.$

Let
$$
\pi :\mathbb{R}^{d-2}\times\mathbb{R}^{2}\longrightarrow\mathbb{R}^{2}
$$
be the trivial vector bundle, where $\pi: (u_{1},...,u_{d-2},x,y)\mapsto (x,y)$.

Denote by $\mathfrak{A}_{1}\subset \mathbf{J}^{1}\left( \pi \right) $  the
subbundle of the $1$-jet bundle corresponding to the abelian equation, and by $\mathfrak{A}_{k}\subset \mathbf{J}^{k}\left( \pi \right) $ the $\left(
k-1\right) $-prolongation of $\mathfrak{A}_{1}.$

Let
$$
\pi _{k,k-1}:\mathfrak{A}_{k}\longrightarrow \mathfrak{A}_{k-1}
$$
be the restrictions of the natural jet projections
$$
\pi _{k,k-1}:\mathbf{J}^{k}\left( \pi \right) \longrightarrow \mathbf{J}^{k-1}\left( \pi \right) .
$$

Then, if $k\leq d-2$, one can easily check that  $\mathfrak{A}_{k}$  are vector bundles, the maps $\pi _{k,k-1}$ are projections and
$$
\dim \ker \pi _{k,k-1}=d-k-2.
$$

In other words, we have the following tower of vector bundles:
\begin{equation*}
\mathbb{R}^2\overset{\pi }{\longleftarrow }\mathbb{R}^{d+2}\overset{\pi _{1,0}}{%
\longleftarrow }\mathfrak{A}_{1}\overset{\pi _{2,1}}{\longleftarrow }%
\mathfrak{A}_{2}\overset{\pi _{3,1}}{\longleftarrow }\cdots \overset{\pi
_{d-3,d-4}}{\longleftarrow }\mathfrak{A}_{d-3}\overset{\pi _{d-2,d-3}}{%
\longleftarrow }\mathfrak{A}_{d-2}.
\end{equation*}%
The last projection
$$
\pi _{k,k-1}:\mathfrak{A}_{d-2}\longrightarrow \mathfrak{A}_{d-3}
$$
is an isomorphism, and geometrically it can be viewed as a linear Cartan connection (see \cite{Ly95}) in the vector bundle
$$
\pi_{d-3}:\mathfrak{A}_{d-3}\rightarrow \mathbb{R}^2.
$$
This proves that \emph{the abelian equation is formally integrable if and only if this linear connection is flat}.

It is easy to see that the dimension of this bundle is equal to $(d-2)(d-1)/2.$

The dimension of the solution space is the rank of the corresponding $d$-web. The above computation shows that the rank of a $d$-web is finite-dimensional and does not exceed
$$
\frac{(d-1)(d-2)}{2}.
$$

This  result was first established by Bol \cite{bo32} (see also \cite{B55}).

The compatibility conditions for the abelian equation can be found (see \cite{GL07}) by use of multi-brackets (see \cite{KL02}).

These conditions have the form
\begin{equation*}
\varkappa =\square _{1}u_{1}+\cdots +\square _{d-2}u_{d-2}=0,
\end{equation*}%
where
\begin{equation*}
\square _{i}=\Delta _{1}\cdots \Delta _{d-2}\cdot \delta _{1}-\Delta
_{1}\cdots \Delta _{i-1}\cdot \delta _{1}\cdot \Delta _{i+1}\cdots \Delta
_{d-2}\cdot \Delta _{i}
\end{equation*}%
are linear differential operators of order not exceeding $d-2$.

Summarizing, we get the following

\begin{theorem}
A $d$-web is of maximum rank if and only if $\varkappa =0$ on $\mathfrak{A}_{d-2}.$
\end{theorem}

Remark that $\varkappa $ can be viewed as a linear function on the vector
bundle $\mathfrak{A}_{d-2},$ and therefore the above theorem imposes $%
(d-1)\left( d-2\right) /2$ conditions on the $d$-web (or on $d-2$ web
functions) in order the web has the maximum rank. A calculation of these
conditions is pure algebraic, and we shall illustrate this calculation below
for planar $3$-,\ $4$- and $5$-webs. Note also
that expressions for $\varkappa $ in the case of general $d$-webs are
extremely cumbersome while for concrete $d$-webs it is not the case.

\section{Rank of $\mathbf{4}$-Webs}

\subsection{The Obstruction}

In order to simplify notations, we put $a_{2}=a$ in the normalization for $4$-webs :
\begin{eqnarray*}
\omega _{1}+\omega _{2}+\omega _{3} =0,\\
a \omega _{1}+\omega _{2}+\omega _{4} =0,
\end{eqnarray*}
and reserve the subscripts for the covariant derivatives of $a.$
Thus, $a_{2}=\delta _{2}\left( a\right) $, etc.

For abelian equations we shall use the functions $u,v$, where $u=u_{1}, v=u_{2}$.

Then the abelian equations have the form
\begin{equation*}
\left( u+av\right) \omega _{1}+\left( u+v\right) \omega _{2}+u\omega
_{3}+v\omega _{4}=0,
\end{equation*}%
where
$$
\lambda _{1}=u+av,\lambda _{2}=u+v,\lambda _{3}=u,\lambda _{4}=v,
$$
and the functions $u$ and $v$ satisfy the equations
\begin{equation*}
\renewcommand{\arraystretch}{1.3}
\begin{array}{ll}
\delta_{1}\left( u\right) -\delta _{2}\left( u\right) =0,\\
\delta_{1}\left(v\right) -\delta _{2}\left( av\right) =0,\\
\delta_{1}\left( u\right) +\delta_{1}\left( v\right) =0.
\end{array}
\renewcommand{\arraystretch}{1}
\end{equation*}

In the case of $4$-webs, the tower of prolongations has the form
\begin{equation*}
\mathbb{R}^{2}\overset{\pi }{\longleftarrow }\mathbb{R}^{4}\overset{\pi _{1,0}}{%
\longleftarrow }\mathfrak{A}_{1}\overset{\pi _{2,1}}{\longleftarrow }%
\mathfrak{A}_{2},
\end{equation*}%
where the isomorphism $\pi_{2,1}:\mathfrak{A}_{2}\rightarrow \mathfrak{A}_{1}$ defines a linear Cartan connection on the three-dimensional vector bundle
$$
\pi _{1}:\mathfrak{A}_{1}\rightarrow \mathbb{R}^2.
$$

In what follows, we use coordinates in the jet spaces adjusted to the Chern connection and weight. Thus, for example,  $u_{k,l}$ stands for the operator $\delta_{1}^{k}\delta_{2}^{l}$.

In these coordinates, the abelian equation takes the following form:
\begin{equation*}
\renewcommand{\arraystretch}{1.3}
\begin{array}{ll}
u_{1}-u_{2}=0,\\
v_{1}-av_{2}-a_{2}v=0,\\
u_{1}+v_{1}=0,
\end{array}
\renewcommand{\arraystretch}{1}
\end{equation*}
and the obstruction
\begin{equation*}
\varkappa =(\Delta _{1}\Delta _{2}\delta _{1}-\delta _{1}\Delta _{1}\Delta
_{2})u+(\Delta _{1}\Delta _{2}\delta _{1}-\Delta _{1}\delta _{1}\Delta _{2})v
\end{equation*}%
equals $$
\varkappa =c_{0}v_{2}+c_{1}v+c_{2}u ,
$$
where $c_0, c_1,$ and $c_2$ are certain functions of the curvature function $K$, the basic
invariant $a$ and their covariant derivatives $K_i$ and $a_i, a_{ij}$ $($see formula $(1)$ in
\cite{GL07}).

The coefficient $c_{0}$ in the expression of $\varkappa $ has an intrinsic
geometric meaning.

Namely, by the \emph{curvature function} of a $4$-web we mean the arithmetic mean of the curvatures of its $3$-subwebs $[1,2,3],\;[1,2,4],\;[1,3,4]$ and $[2,3,4]$.

Then (see \cite{GL07}) the coefficient $c_{0}$ equals the curvature function of the $4$-web.

\subsection{$\mathbf{4}$-Webs of Maximum Rank}

A planar $4$-web has the maximum rank three if and only if the obstruction $%
\varkappa $ identically equals zero, i.e., if and only if $c_{0}=c_{1}=c_{2}=0.$
Computing these coefficients leads us to the following result (see \cite{GL07}).

\begin{theorem}
\label{4webmaxrank}A planar $4$-web is of maximum rank if and
only if  the following relations hold:
\begin{equation*}
\renewcommand{\arraystretch}{2}
\begin{array}{lll}
c_{0} &=&K+\displaystyle\frac{a_{11}-aa_{22}-2\left( 1-a\right) a_{12}}{4a(1-a)}+\frac{%
\left( -1+2a\right) a_{1}^{2}-a^{2}a_{2}^{2}+2\left( 1-a\right)
^{2}a_{1}a_{2}}{4\left( 1-a\right) ^{2}a^{2}}, \\
c_{1} &=&\displaystyle\frac{K_{2}-K_{1}}{4(1-a)}+\frac{\left( a-4\right) a_{1}+\left(
11-20a+12a^{2}\right) a_{2}}{12\left( 1-a\right) ^{2}a}K+\frac{%
a_{112}-a_{122}}{4a(1-a)} \\
&&+\displaystyle\frac{a_{1}-aa_{2}}{4a^{2}(1-a)}a_{22}+\frac{\left( 2a-1\right) \left(
a_{1}-aa_{2}\right) }{4\left( 1-a\right) ^{2}a^{2}}a_{12}-\frac{%
a_{2}^{2}\left( \left( 1-2a\right) a_{1}+aa_{2}\right) }{4\left( 1-a\right)
^{2}a^{2}}, \\
c_{2} &=&\displaystyle\frac{aK_{2}-K_{1}}{4a(1-a)}+\frac{\left( 1-2a\right) a_{1}-\left(
a-2\right) aa_{2}}{4\left( 1-a\right) ^{2}a^{2}}K.
\end{array}
\renewcommand{\arraystretch}{1}
\end{equation*}
\end{theorem}

Vanishing of the coefficients $c_{1}$ and $c_{2}$ for $4$-webs with $c_{0}=0$ is equivalent to linearizability of the web (see \cite{AGL04}). Therefore the above theorem can be formulated in pure geometric terms:

\begin{theorem}
\label{4webmaxrank3}A $4$-web is of maximum rank three if and only if it
is linearizable and its curvature vanishes.
\end{theorem}

Theorem \ref{4webmaxrank3} leads to interesting results in web geometry.
\begin{enumerate}
\item  A linearizable planar 4-web is of
maximum rank if and only if its curvature vanishes.
\item  A planar 4-web of maximum rank is linearizable (algebraizable) (Poincar\'{e}).
\item  If a planar 4-web\ with a constant basic invariant $a$ has maximum rank, then it is parallelizable.
\item Parallelizable planar $4$-webs have maximum rank.
\item The Mayrhofer  4-webs are of maximum rank.
\end{enumerate}

Recall that a 4-web is called the \emph{Mayrhofer} web if all $3$-subwebs of this web are parallelizable.

\subsection{4-Webs of Maximum Rank and Surfaces of Double Translation}

A surface $S\subset\mathbb{R}^{3}$ is a \emph{surface of translation} in if it admits a vector parametric representation $r=R(u,v)$, where $R(u,v)$ is a solution of the wave equation
$$
R_{uv}=0.
$$
Then
\begin{equation}
r=f(u)+g(v),
\label{1st repr}
\end{equation}
or, in components of vectors
\begin{equation*}
\renewcommand{\arraystretch}{1.3}
\left\{
\begin{array}{ll}
x = f^1 (u) + g^1 (v),\\
y = f^2 (u) + g^2 (v),\\
z= f^3 (u) + g^3 (v).
\end{array}
\right.
\renewcommand{\arraystretch}{1}
\end{equation*}

A surface $S$ is a \emph{surface of double translation}
if in addition to representation (\ref{1st repr}) it
also admits a representation
\begin{equation}
r=h(s)+k(t),
\label{2nd repr}
\end{equation}
such that the coordinate functions  $u,v,s,t$ on the surface are pairwise independent.

In other words, they define a 4-web on the surface $S$.
If $S$ is a surface of double translation, then it follows from  (\ref{1st repr}) and (\ref{2nd repr}) that
\begin{equation}
f^i (u) + g^i (v) - h^i (s) - k^i (t)=0
\label{double transl}
\end{equation}
for $i=1,2,3.$
These relations can be viewed as abelian relations for the 4-web mentioned above.

If the surface $S$ does not belong to a plane, then (\ref{double transl}) gives three independent abelian relations for the web. Therefore, this web has the maximal rank, and as we have seen earlier, it is linearizable (algebraizable). This result was first proved by Sophus Lie in the form.

\begin{theorem} {\rm(\cite{Lie82})}
If $S$ is a surface of double translation not belonging to a plane, then the curves $f' (u), g'(v), h'(s)$ and $k'(t)$ belong to an algebraic curve of degree four.
\end{theorem}

More on the subject, its further developments and references one can find in \cite{C82} and \cite{AG00}.

\subsection{$\mathbf{4}$-Webs of Rank Two}

As we have seen, a $4$-web admits an abelian equation (has a
positive rank) if and only if the equation
\begin{equation}
c_{0}v_{2}+c_{1}v+c_{2}u=0  \label{eq4}
\end{equation}%
has a nonzero solution.

Suppose that  $c_{0}=0$. Then if two other coefficients $c_{1}=c_{2}=0$,
then a $4$-web is of maximum rank three. If $c_{0}=0$ but one of the coefficients $%
c_{1}$ or $c_{2}$ is not $0,$ then $c_{1}v+c_{2}u=0$ and then,
say $u,$ satisfies a first-order PDE system of two equations. Therefore,
the $4$-web admits not more than one abelian equation (i.e., it is of
rank one or zero).

Assume that $c_{0}\neq 0.$  Then we can find all first derivatives $u_{i},v_{j}$ from the abelian equation and  (\ref{eq4}):

\begin{equation*}
\renewcommand{\arraystretch}{2}
\begin{array}{ll}
u_{1}=&\displaystyle\frac{a c_{1}-a_{2}c_{0}}{c_{0}}v+\frac{ac_{2}}{c_{0}}u,\\
u_{2}=&\displaystyle\frac{a c_{1}-a_{2}c_{0}}{c_{0}}v+\frac{ac_{2}}{c_{0}}u,\\
v_{1}=&\displaystyle\frac{a_{2}c_{0}-a c_{1}}{c_{0}}v-\frac{ac_{2}}{c_{0}}u,\\
v_{2}=&-\displaystyle\frac{c_{1}}{c_{0}}v-\frac{c_{2}}{c_{0}}u.
\end{array}
\renewcommand{\arraystretch}{1}
\end{equation*}
Therefore, a 4-web has rank two if and only if the above system is compatible.

\begin{theorem}
\label{rank2}A planar $4$-web is of rank two if and only if $c_{0}\neq 0,$
and
\begin{equation}
G_{ij}=0, \;i,j=1,2,  \label{cond for rank 2}
\end{equation}%
where%
\begin{eqnarray*}
G_{11} &=&ac_{0}(c_{2,2}-c_{2,1})+ac_{2}(c_{0,1}-c_{0,2})-a\left( 1-a\right)
c_{1}c_{2} \\
&&+\left( 2a_{2}-a_{1}-aa_{2}\right) c_{0}c_{2}-Kc_{0}^{2}, \\
G_{12} &=&ac_{0}(c_{1,2}-c_{1,1})+ac_{1}(c_{0,1}-c_{0,2})-a\left( 1-a\right)
c_{1}^{2} \\
&&+\left( 2a_{2}-a_{1}-2aa_{2}\right) c_{0}c_{1}+\left(
a_{2}^{2}+a_{12}-a_{22}\right) c_{0}^{2}, \\
G_{21}
&=&c_{0}(c_{2,1}-ac_{2,2})+c_{2}(ac_{0,2}-c_{0,1})-2a_{2}c_{0}c_{2}+a\left(
1-a\right) c_{2}^{2}, \\
G_{22} &=&c_{0}(c_{1,1}-ac_{1,2})+c_{1}(ac_{0,2}-c_{0,1})+a\left( 1-a\right)
c_{1}c_{2}-a_{2}c_{0}c_{1} \\
&&-a_{2}(1-a)c_{0}c_{2}+\left( a_{22}-K\right) c_{0}^{2}.
\end{eqnarray*}
\end{theorem}

\begin{example}
\label{rank2nonlin2}Consider the planar $4$-web with the following web functions
\begin{equation*}
x,\;y,\;\displaystyle\frac{x}{y},\;xy(x+y).
\end{equation*}

The linearizability conditions (see \cite{AGL04}) for this web are not satisfied, and therefore, this $4$-web is not linearizable, but in this case $G_{11}=G_{12}=G_{21}=G_{22}=0.$
Hence, the $4$-web is of rank two.
\end{example}

This example  gives us the following important property:

\emph{General $4$-webs of rank two are not
linearizable.}

\subsection{$\mathbf{4}$-Webs of Rank One}

As we have seen earlier, a $4$-web can be of rank one if $c_{0}=0$ but
one of the coefficients $c_{1}$ and $c_{2}$ of (\ref{eq4}) is not $0$ or
if $c_{0}\neq 0.$ The following theorem outlines the four cases when a
$4$-web can be of rank one.

\begin{theorem}
\label{rank1}A planar $4$-web is of rank one if and only if one of the
following conditions holds:

\begin{enumerate}
\item $c_{0}=0,$ $J_{1}=J_{2}=0,$ where%
\begin{eqnarray*}
J_{1} &=&a_{2}c_{1}c_{2}(c_{1}-c_{2})+ac_{2}^{2}(c_{1,2}-c_{1,1}) \\
&&+c_{1}c_{2}(c_{1,1}+a(c_{2,1}-c_{1,2}-c_{2,2}))+c_{1}^{2}(ac_{2,2}-c_{2,1}),
\\
J_{2} &=&c_{1}^{2}\left( c_{1}-c_{2}\right) ^{2}K+\left(
c_{1,11}-c_{1,12}\right) c_{1}c_{2}\left( c_{2}-c_{1}\right) \\
&&+c_{1}^{2}\left( c_{1}-c_{2}\right) \left( c_{2,11}-c_{2,12}\right)
-c_{2}\left( 2c_{1}-c_{2}\right) c_{1,1}(c_{1,2}-c_{1,1}) \\
&&+c_{1}^{2}c_{2,1}(c_{1,2}-c_{2,2}+c_{2,1})+c_{1}^{2}c_{1,1}(c_{2,2}-2c_{2,1})
\end{eqnarray*}%
and $c_{1}\neq c_{2},\;c_{1}\neq 0.$

\item $c_{0}=0,c_{1}=c_{2}\neq 0,$ and $J_{3}=0,$ where%
\begin{equation*}
J_{3}=\left( a_{22}-a_{12}\right) \left( 1-a\right)
+a_{2}(a_{2}-a_{1})-\left( 1-a\right) ^{2}K.
\end{equation*}

\item $c_{0}=0,c_{1}=0,c_{2}\neq 0,$ and $J_{4}=0,$ where
\begin{equation*}
J_{4}=a_{12}a-a_{1}a_{2}-Ka^{2}.
\end{equation*}

\item $c_{0}\neq 0,$ and $J_{10}=J_{11}=J_{12}=0,$ where
\begin{equation*}
J_{10}=G_{11}G_{22}-G_{21}G_{12},
\end{equation*}%
\begin{eqnarray*}
J_{11} &=&c_{0}(G_{21,1}G_{22}-G_{22,1}G_{21})+(a_{2}c_{0}-ac_{1})G_{21}^{2}
\\
&&+(ac_{2}-a_{2}c_{0}+ac_{1})G_{21}G_{22}-ac_{2}G_{22}^{2},
\end{eqnarray*}%
\begin{eqnarray*}
J_{12} &=&c_{0}(G_{21,2}G_{22}-G_{22,2}G_{21})+(a_{2}c_{0}-ac_{1})G_{21}^{2}
\\
&&+a(c_{2}-c_{1})G_{21}G_{22}-c_{2}G_{22}^{2}.
\end{eqnarray*}
\end{enumerate}
\end{theorem}

\begin{proof}
See \cite{GL07}.
\end{proof}

\begin{example}
\label{rank1nonlin1} Consider the planar $4$-web with the following web functions
\begin{equation*}
x,\;y,\;\displaystyle\frac{xy^{2}}{(x-y)^{2}},\;\displaystyle\frac{x^{2}y}{(x-y)^{2}}.
\end{equation*}

For this web, we have  $c_{0}=0$ and $J_{1}=J_{2}=0.$ Thus, we have the web of
type $1$ as indicated in Theorem \ref{rank1}, and this $4$-web is of rank
one.
\end{example}

In this example the 4-web is not linearizable. Therefore,

\emph{General $4$-webs of rank one are not linearizable.}

\section{Planar $\mathbf{5}$-Webs  of Maximum Rank}

Let us consider a planar $5$-web in the standard normalization%
\begin{equation*}
\omega _{1}+\omega _{2}+\omega _{3} =0, \;\;
a\omega _{1}+\omega _{2}+\;\omega_{4} =0, \;\;
b\omega _{1}+\omega _{2}+\;\omega_{5} =0,
\end{equation*}%
where $a$ and $b$ are the basic invariants of the web.

The abelian equation for such a web has the form
\begin{equation*}
(w+au+bv)\omega _{1}+(w+u+v)\omega _{2}+w\omega _{3}+u\omega _{4}+v\omega
_{5}=0,
\end{equation*}%
where we have
$$
\lambda _{1}=w+au+bv,\;\;\lambda _{2}=w+u+v,
$$
and
$$
\lambda_{3}=w,\;\;\lambda _{4}=u, \;\; \lambda _{5}=v.
$$

The functions $w, u,$ and $v$ satisfy the abelian equation%
\begin{equation*}
\renewcommand{\arraystretch}{1.3}
\begin{array}{ll}
\delta _{1}\left( w\right) -\delta _{2}\left( w\right) =0, &
\;\; \delta _{1}\left( u\right) -\delta _{2}\left( au\right) =0, \\
\delta _{1}\left( v\right) -\delta _{2}\left( bv\right) =0, &
\;\; \delta _{1}\left( w\right) +\delta _{1}\left( u\right) +\delta _{1}\left(
v\right) =0,%
\end{array}
\renewcommand{\arraystretch}{1}
\end{equation*}
and their compatibility condition takes the form%
\begin{equation*}
\renewcommand{\arraystretch}{1.3}
\begin{array}{ll}
\varkappa =&\left( \Delta _{1}\Delta _{2}\Delta _{3}\delta _{1}-\delta
_{1}\Delta _{2}\Delta _{3}\Delta _{1}\right) \left( w\right) +\left( \Delta
_{1}\Delta _{2}\Delta _{3}\delta _{1}-\Delta _{1}\delta _{1}\Delta
_{3}\Delta _{2}\right) \left( u\right) \\
&+\left( \Delta _{1}\Delta _{2}\Delta _{3}\delta _{1}-\Delta _{1}\Delta
_{2}\delta _{1}\Delta _{3}\right) \left( v\right) =0.%
\end{array}
\renewcommand{\arraystretch}{1}
\end{equation*}

In the canonical coordinates in the jet bundles, the abelian equation has
the form
\begin{equation*}
\begin{array}{ll}
u_{1}+v_{1}+w_{1}=0, &\;\; v_{1}-bv_{2}-b_{2}v=0, \\
u_{1}-au_{2}-a_{2}u=0, &
\;\; w_{1}-w_{2}=0,
\end{array}%
\end{equation*}%
and the obstruction $\varkappa $ equals%
\begin{equation*}
c_{0}w_{22}+c_{1}w_{2}+c_{2}v_{2}+c_{3}w+c_{4}u+c_{5}v=0,
\end{equation*}%
where the explicit form of expressions for the coefficients $c_9, c_1, c_2, c_3,
c_4$ and $c_5$ can be found in \cite{GL07}.

This gives the following result \cite{GL07}:

\emph{A planar $5$-web is of maximum rank if and only if the invariants $%
c_{0},c_{1},c_{2},c_{3},c_{4}\ $and $c_{5}$ vanish.}

Similar to the case of $4$-webs, the coefficient $c_{0}$ in the expression
of $\varkappa $ for $5$-webs has an intrinsic geometric meaning.

Namely, we call by the \emph{curvature function} of  a $5$-web the arithmetic mean of
the curvature functions of its ten $3$-subwebs.

The straightforward calculation shows that the curvature function equals to $c_{0}.$

In other words \cite{GL07}, the curvature of a planar $5$-web of maximum rank equals zero.

For the case of  planar $5$-webs with constant basic invariants $a$ and $b$ the invariants $c_{i},$ for $i=0,1,2,3,4,5,$ vanish (and the web is of maximum rank) if and
only if this web is parallelizable \cite{GL07}.

\begin{example}
\label{Bol 5-web} We consider the Bol $5$-web with the web functions

\begin{equation*}
x,y,\displaystyle\frac{x}{y},\displaystyle\frac{1-y}{1-x},\displaystyle\frac{x-xy}{y-xy}.
\end{equation*}%

For this web we have
$$
K=0, \,a=\displaystyle\frac{xy-x}{xy-y}, \,b=\displaystyle\frac{y-1}{x-1},
$$
and $c_{i}=0$, for $i=0,1,...,5$.

Thus, the $3$-web is of maximum rank.

Using the linearizability conditions for planar $5$-webs \cite{AGL04},
 we see that the Bol $5$-web is not linearizable.
\end{example}

The above example leads us to the following important observation:

\emph{General planar $5$-webs of maximum rank are not linearizable.}

\section{Projective Structures and Planar 4-Webs}
In this section we give more direct construction of the projective structure associated with 4-webs (see \cite{GL08a}).

 Remind that an affine connection $\nabla$ on the plane determines a covariant differential
 $$
 d_{\nabla }:\Omega ^{1}\left( \mathbb{R}^{2}\right) \rightarrow \Omega
^{1}\left( \mathbb{R}^{2}\right) \otimes \Omega ^{1}\left( \mathbb{R}%
^{2}\right).
 $$
 This differential splits into the sum
 $$
 d_{\nabla }=d_{\nabla }^{a}\oplus d_{\nabla }^{s},
 $$
where
$$
d_{\nabla }^{a}:\Omega ^{1}\left( \mathbb{R}^{2}\right) \rightarrow \Omega
^{2}\left( \mathbb{R}^{2}\right)
$$
is the skew-symmetric part, and
$$
d_{\nabla }^{s}:\Omega ^{1}\left( \mathbb{R}^{2}\right) \rightarrow S
^{2}\left( \Omega^{1}\right)\left( \mathbb{R}^{2}\right)
$$
is the symmetric one.

The connection is torsion-free if and only the skew-symmetric part coincides with the de Rham differential: $$d_{\nabla }^{a}=d.$$

A foliation given by a differential 1-form $\omega$ is geodesic (i.e., all leaves of the foliation are geodesics) with respect to connection $\nabla$ if and only if (see \cite{GL08a}):
$$
d_{\nabla }^{s}(\omega)=\theta \cdot \omega
$$
for some differential 1-form $\theta$.

Remark that it follows from the above formula that two affine connections, say, $\nabla$ and $\nabla'$, are projectively equivalent (i.e., have the same geodesics) if and only if there exists a differential 1-form $\rho$ such that
$$
d_{\nabla }^{s}(\omega)-d_{\nabla '}^{s}(\omega)=\rho \cdot \omega
$$
for all differential 1-forms $\omega$.

Assume that a 4-web is given by differential 1-forms $\omega_{i}, i=1, 2, 3, 4,$ which are normalized
\begin{eqnarray*}
\omega_{1}+\omega_{2}+\omega_{3}=0,\\
a \omega_{1}+\omega_{2}+\omega_{4}=0,
\end{eqnarray*}
and
$$
d \omega_{3}=0.
$$

Let $\nabla$ be a torsion-free connection for which all foliations $\omega_{i}=0,\, i=1, 2, 3, 4,$ are geodesics.

We call such 4-webs \emph{geodesic}.

Then
$$
d_{\nabla }^{s}(\omega_{i})=\theta_{i} \cdot \omega_{i}
$$
for all $i=1, 2, 3, 4$ and some differential 1-forms $\theta_{i}$.

Differentiating the normalization conditions, we get
\begin{eqnarray*}
\theta_{1}\cdot\omega_{1}+\theta_{2}\cdot\omega_{2}+\theta_{3}\cdot\omega_{3}=0,\\
da\cdot\omega_{1}+ a \theta_{1}\cdot\omega_{1}+\theta_{2}\cdot\omega_{2}+\theta_{4}\cdot\omega_{4}=0.
\end{eqnarray*}
If
$$
\theta_{i}=A_{i}\omega_{1}+B_{i}\omega_{2}
$$
for all $i=1, 2, 3, 4,$ then the above system is just a system of linear equations for coefficients $A_{i}$ and $B_{i}$.

Solving this system, we find that
\begin{equation*}
\renewcommand{\arraystretch}{1.5}
\begin{array}{ll}
A_{2}=A_{1}+z,& B_{2}=B_{1}+z,\\
A_{3}=A_{1},& B_{3}=B_{1}+z,\\
A_{4}=A_{1}+ \displaystyle\frac{a_{1}}{a},&  B_{4}=B_{1}+z,
\end{array}
\renewcommand{\arraystretch}{1}
\end{equation*}
where
$$
z=\displaystyle\frac{a_{1}-a a_{2}}{a(1-a)}.
$$
In other words, the affine connection is completely determined by the differential 1-form $\theta_{1}$, and
\begin{equation*}
\renewcommand{\arraystretch}{1.5}
\begin{array}{ll}
&\theta_{2}=\theta_{1}-z \omega_{3},\\
&\theta_{3}=\theta_{1}+z \omega_{2},\\
&\theta_{4}=\theta_{1}+\displaystyle\frac{a_{1}}{a}\omega_{1}+z \omega_{2}.
\end{array}
\renewcommand{\arraystretch}{1}
\end{equation*}
This shows that all such affine connections are projectively equivalent.
Taking the representative with
$$
\theta_{1}=\displaystyle\frac{z}{2} \omega_{3},
$$
we get the following result (this result was first obtained 
in \cite{BB38}, \S29, p. 246):
\begin{theorem}
There is a unique projective structure associated with a planar $4$-web in such a way that the $4$-web is geodesic with respect to the structure.

The projective structure is an equivalence class of the torsion-free affine connection $\nabla$ with the following symmetric differential:
\begin{equation*}
\renewcommand{\arraystretch}{1.5}
\begin{array}{ll}
&d_{\nabla }^{s}(\omega_{1})=\displaystyle\frac{z}{2}\ \omega_{3} \cdot \omega_{1},\\
&d_{\nabla }^{s}(\omega_{2})=-\displaystyle\frac{z}{2}\ \omega_{3} \cdot \omega_{2}.
\end{array}
\renewcommand{\arraystretch}{1}
\end{equation*}
\end{theorem}

We say that a planar $d$-web is \emph{geodesic} with respect to an affine connection if all leaves of all foliations are geodesic.

The above theorem gives a criterion for a $d$-web to be geodesic. For simplicity we take the case of 5-webs.
Let a 5-web be given by differential 1-forms $\omega_{i}, i=1, 2, 3, 4 ,5,$ which are normalized as follows:
\begin{equation*}
\renewcommand{\arraystretch}{1.5}
\begin{array}{ll}
\omega_{1}+\omega_{2}+\omega_{3}=0,\\
a \omega_{1}+\omega_{2}+\omega_{4}=0,\\
b \omega_{1}+\omega_{2}+\omega_{5}=0,
\end{array}
\renewcommand{\arraystretch}{1}
\end{equation*}
and
$$
d\omega_{3}=0.
$$
This web is geodesic if and only if the fifth foliation $\omega_{5}=0$ is geodesic with respect to the canonical projective structure determined by the 4-web ($\omega_{i}, i=1, 2, 3, 4$).

We have
$$
d_{\nabla }^{s}(\omega_{5})=\theta_{5} \cdot \omega_{5}+(z b (1-b)-(b_{1}-b b_{2}))\omega_{1}^{2}.
$$
Therefore, in order to have a  geodesic 5-web, the last term should vanish.
\begin{theorem}
A $5$-web is geodesic if and only if the basic invariants $a$ and $b$ satisfy the following condition:
\begin{equation}
\frac{a_{1}-a a_{2}}{a(1-a)}=\frac{b_{1}-b b_{2}}{b(1-b)}. \label{geod cond}
\end{equation}
\end{theorem}

 The linearizability problem (see \cite{AGL04}) for planar webs can be reformulated now as follows: \emph{a planar $d$-web is linearizable if and only if the web is geodesic and the canonical projective structure of one of its $4$-subwebs is flat.}

 The flatness of a projective structure can be checked by the Liouville tensor (see \cite{Lio89}, \cite{Lie83}, \cite{K08}).
 This tensor can be constructed as follows (see, for example, \cite{NS94}). Let $\nabla$ be a representative of the canonical projective structure, and $Ric$ be the Ricci tensor of the connection $\nabla$. Define a new tensor $\mathfrak{P}$ as
 $$
\mathfrak{P}(X,Y)=\frac{2}{3}Ric(X,Y)+\frac{1}{3}Ric(Y,X)
 $$
for all vector fields $X,Y$.

The Liouville tensor $\mathfrak{L}$ is defined as follows:
$$
\mathfrak{L}(X,Y,Z)=\nabla_{X}(\mathfrak{P})(Y,Z)-\nabla_{Y}(\mathfrak{P})(X,Z)
$$
for all vector fields $X,Y,Z$.

The tensor is skew-symmetric in $X$ and $Y$, and therefore it belongs to
$$
\mathfrak{L}\in\Omega^{1}(\mathbb{R}^{2})\otimes\Omega^{2}(\mathbb{R}^{2}).
$$
It is known (see \cite{Lio89}, \cite{NS94}, \cite{Lie83}, \cite{K08}) that \emph{the Liouville tensor depends on the projective structure defined by $\nabla$ and vanishes if and only if the projective structure is flat.}

For the case of the projective structure associated with a planar 4-web we shall call this tensor the \emph{Liouville tensor} of the 4-web

Consider three invariants:
\begin{equation}
w=\frac{f_y}{f_x}, \;\;  \alpha=\frac{a a_y-w a_x}{w a(1-a)}, \;\;  k = (\log w)_{xy}.
\label{invar}\end{equation}

Then the Liouville tensor has the form
$$
\mathfrak{L}=(L_{1}\omega_{1}+\frac{L_{2}}{w}\omega_{2})\otimes\omega_{1}\wedge\omega_{2},
$$
where $L_{1}$ and $L_{2}$ are relative differential invariants of order three.
 The explicit formulas for these invariants are
  \begin{equation}
  \renewcommand{\arraystretch}{1.5}
  \begin{array}{lll}
 3L_{1}&=& w (-(k w)_x + \alpha_{xx} + \alpha \alpha_x) + (\alpha w_{xx} + (\alpha^2 +3 \alpha_x) w_x -2\alpha_{xy} - 2 \alpha\alpha_y)\\
  && +w^{-1} (-\alpha w_{xy} - 2\alpha_y  w_x + \alpha w_{x}^2) +  w^{-2}\alpha w_x w_y,\\
 3L_{2}&=&w^2(-(kw^{-1})_y + 2\alpha \alpha_x) + w (2\alpha^2 w_x - 2 \alpha_{xy} - \alpha \alpha_y) \\
 &&+ (-\alpha w_{xy} - 2 \alpha_y w_x + \alpha_{yy}) +w^{-1} (\alpha w_x w_y - \alpha_y w_y).
 \end{array}
 \renewcommand{\arraystretch}{1} \label{Lio tensor}
 \end{equation}

 Summarizing, we get the following result.
 \begin{theorem} \label{linthm}
 A planar d-web is linearizable if and only if the web is geodesic and the Liouville tensor of one of its $4$-subwebs vanishes.
 \end{theorem}

 \begin{corollary} If the basic invariants of all $4$-subwebs of  a $d$-web
are constants, then  the $d$-web is linearizable if and only if it is
 parallelizable.
 \end{corollary}
 \begin{proof}
First of all, the web is geodesic because of conditions (\ref{geod cond}).

Moreover, for a 4-subweb , condition  $a=\operatorname{const}.$ implies $\alpha = 0$, and by
 Theorem \ref{linthm} and (\ref{Lio tensor}), the 4-web is linearizable if and
 only if
 $$(k w)_x = 0, $$
 and
 $$(k w^{-1})_y = 0. $$

 Then $w = A(x) B(y)$ and by (\ref{3-web curvature}), $K=0$. Therefore, due to  Section 3, the 4-web is parallelizable.

 The $d$-web is parallelizable too, because it geodesic and has constant basic invariants.
 \end{proof}

\end{document}